\documentclass[smallextended]{svjour3}                     % onecolumn (standard format)

\smartqed  % flush right qed marks, e.g. at end of proof
\usepackage{graphicx}
\usepackage{fix-cm}
\usepackage{geometry}                % See geometry.pdf to learn the layout options. There are lots.
\usepackage{graphicx}
\usepackage{amssymb}
\usepackage{epstopdf}
\usepackage{pgf,tikz}
\usepackage{latexsym}
\usepackage{amsmath,amsfonts}
\usepackage{graphicx}
\usepackage{geometry}
\usepackage{graphicx}
\usepackage{amssymb}
\usepackage{epstopdf}
\usepackage{pgf,tikz}
\usepackage{latexsym}
\usepackage{color}
\usepackage{xcolor}
\usepackage{soul}
\newcommand{\ignore}[1]{}

\newcommand{\oR}{{\mathbb R}}
\newcommand{\oN}{{\mathbb N}}

\journalname{}
\begin{document}

\title{An alternative proof of a PTAS for fixed-degree polynomial optimization over the simplex}


\author{Etienne de Klerk         \and
        Monique Laurent          \and
        Zhao Sun}

\institute{Etienne de Klerk \at
              Tilburg University \at
              PO Box 90153, 5000 LE Tilburg, The Netherlands \\
              \email{E.deKlerk@uvt.nl}
           \and
           Monique Laurent \at
              Centrum Wiskunde \& Informatica (CWI), Amsterdam and Tilburg University \at
              CWI, Postbus 94079, 1090 GB Amsterdam, The Netherlands\\
              \email{monique@cwi.nl}
           \and
           Zhao Sun \at
              Tilburg University \at
              PO Box 90153, 5000 LE Tilburg, The Netherlands \\
              \email{z.sun@uvt.nl}
}

\date{Received: date / Accepted: date}

\maketitle

\begin{abstract}
{
The problem of minimizing a polynomial over the standard simplex is one of the basic
NP-hard nonlinear optimization problems --- it contains the maximum clique problem in graphs
as a special case.
It is known that the problem allows a polynomial-time approximation scheme (PTAS)
for polynomials of fixed degree, which is based on polynomial evaluations at the points of a sequence of regular grids.
In this paper, we provide an alternative proof of the PTAS property.
The proof relies on the properties of Bernstein approximation on the simplex.
We also refine a known error bound for the scheme for polynomials of degree three.
The main contribution of the paper is to provide new insight into the PTAS by establishing precise links with Bernstein approximation and the multinomial distribution.
}

{
% \PACS{PACS code1 \and PACS code2 }
}
\end{abstract}

\section{Introduction and preliminaries}
{\subsection{Polynomial optimization over the simplex}
Let $\mathcal{H}_{n,d}$ denote the set of all homogeneous polynomials of degree $d$ in $n$ variables.
 We consider the problem of minimizing a polynomial $f\in\mathcal{H}_{n,d}$ on the standard simplex
\[
\Delta_n=\left\{x\in\oR_+^n:\sum_{i=1}^nx_i=1\right\}.
\]
That is, the problem of computing
$$\underline{f}=\min_{x\in\Delta_n}f(x), \ \text{ or } \ \overline{f}=\max_{x\in\Delta_n}f(x).$$
}

{
This problem is known to be NP-hard, even if $f$ is a quadratic function. Indeed, if $G$ denotes a graph
with vertex set $V$ and adjacency matrix $A$, and $I$ denotes the identity matrix, then
the maximum cardinality $\alpha(G)$ of a stable set in $G$ can be obtained via
$${1\over \alpha(G)}=\min_{x\in \Delta_{|V|}}x^T(I+A)x,$$
by a theorem of Motzkin and Straus \cite{MS}.
}

{
On the other hand, the problem does allow a polynomial-time approximation scheme (PTAS), as was shown by Bomze and De Klerk (for quadratic $f$)
\cite{BK02}, and by De Klerk, Laurent and Parrilo (for more general, fixed-degree $f$) \cite{KLP06}.
The PTAS is particularly simple, and takes the minimum of $f$ on the regular grid:
\[
\Delta(n,r)=\{x\in\Delta_n:rx\in\oN^n\}
\]
for increasing values of $r$.
We denote the minimum over the grid by
\begin{equation*}
f_{\Delta(n,r)}=\min_{x\in\Delta(n,r)}f(x),
\end{equation*}
and observe that the computation of $f_{\Delta(n,r)}$ requires $|\Delta(n,r)| = {n+r-1 \choose r}$
evaluations of $f$. %, \tcolred{thus polynomial in $n$ for fixed $r$.}

Several properties of the regular grid $\Delta(n,r)$ have been studied in the literature.
In Bos \cite{Bos}, the Lebesgue constant of $\Delta(n,r)$ is studied in the context of Lagrange interpolation and finite element methods.
Given a point $x \in \Delta_n$, Bomze, Gollowitzer and Yildirim \cite{BGY} study a scheme to find the closest point to $x$ on $\Delta(n,r)$
 with respect to certain norms (including $\ell_p$-norms for finite $p$).
 Furthermore, for any quadratic polynomial $f \in \mathcal{H}_{n,2}$ and $r\ge 2$,
 Sagol and Yildirim \cite{SY13} and Yildirim \cite{YEA12} consider the upper bound on $\underline{f}$ defined
by $\min_{x\in\cup_{k=2}^{r}\Delta(n,k)} f(x)$ $(r = 2,3,\ldots)$, and analyze the error bound.}
The following error bounds are known for the approximation $f_{\Delta(n,r)}$ of $\underline{f}$.

\begin{theorem}[(i) Bomze-De Klerk \cite{BK02} and (ii) De Klerk-Laurent-Parrilo \cite{KLP06}]
\label{thm:PTAS}
\begin{itemize}
\item[(i)]
For any quadratic polynomial $f \in \mathcal{H}_{n,2}$ and $r\ge 2$,
one has $$f_{\Delta(n,r)}-\underline{f} \le {\overline{f}-\underline{f}\over r}.$$
\item[(ii)]
For any  polynomial $f \in \mathcal{H}_{n,d}$ and $r\ge d$, one has
\[
f_{\Delta(n,r)}-\underline{f} \le  \left(1-{r^{\underline{d}}\over r^d}\right){2d-1\choose d}d^d(\overline{f}-\underline{f}),
\]
where $r^{\underline{d}} = r(r-1)\cdots(r-d+1)$.
\end{itemize}
\end{theorem}

\noindent Note that these results indeed imply the existence of a PTAS in the sense of the following definition,
that has been used by several authors (see e.g. \cite{BR,KHE08,KLP06,NWY,Va}).
\begin{definition}[PTAS]
A value $\psi_{\epsilon}$ approximates $\underline{f}$ with relative accuracy $\epsilon\in[0,1]$ if
$$|\psi_{\epsilon}-\underline{f}|\le\epsilon(\overline{f}-\underline{f}).$$
The approximation is called implementable if $\psi_{\epsilon}=f(x_{\epsilon})$ for some feasible $x_{\epsilon}$.
If a problem allows an implementable approximation $\psi_{\epsilon}=f(x_{\epsilon})$ for each $\epsilon\in(0,1]$,
such that the feasible $x_{\epsilon}$ can be computed in time polynomial in $n$ and the bit size required to represent $f$, we say that the problem allows a polynomial time approximation scheme (PTAS).
\end{definition}

Indeed, Theorem \ref{thm:PTAS} clearly implies that $f_{\Delta(n,r)}$ yields a PTAS for polynomials of fixed degree $d$.
In this paper we give alternative proofs of this result, and also refine the relevant error bound in the special case of degree three polynomials.
The proof of the PTAS in the quadratic case is completely elementary, and much simpler than the proof given in \cite{BK02}.
It is in fact closely related to a proof given by Nesterov \cite{Nes2003}; see Section \ref{sec:conclusion}.
In fact, the main contribution of our paper is to provide new insight into the PTAS by establishing precise connections with
Bernstein approximation, and the approach of Nesterov \cite{Nes2003}, which in turn requires an understanding of the
 precise connection to the multinomial distribution.
We also prove, by giving an example, that the error bound in Theorem \ref{thm:PTAS}(i) is tight in terms of its dependence on $r$.

Our main tool will be Bernstein approximation on the simplex {(which is similar to the approach used by De Klerk and Laurent \cite{KL} for polynomial optimization over the hypercube)}.

We start by reviewing the necessary background material on Bernstein approximation.

\subsection{Bernstein approximation on the simplex}\label{sec1}
Given an integer $r\ge0$, we define
$$I(n,r):=\left\{\alpha\in\oN^n: \sum_{i=1}^n \alpha_i =r\right\} = r\Delta(n,r).$$
The Bernstein approximation of order $r\ge 1$ on the simplex of a polynomial $f\in \mathcal{H}_{n,d}$  is the polynomial  $B_r(f)\in \mathcal{H}_{n,r}$ defined by%, where $r\ge d$, is

\begin{equation}\label{bernstein-approximation}
B_r(f)(x)=\sum_{\alpha\in I(n,r)}f\left({\alpha\over r}\right){r!\over \alpha!}x^{\alpha}, % \quad \quad (x \in \Delta_n),
\end{equation}
where  $\alpha!:=\prod_{i=1}^n \alpha_i!$ and $x^{\alpha}:=\prod_{i=1}^nx_i^{\alpha_i}$.
For instance, for the constant polynomial $f\equiv 1$, its Bernstein approximation of any order $r$ is
$\sum_{\alpha\in I(n,r)} {r!\over \alpha!} x^\alpha$,  which is equal to $ (\sum_{i=1}^n x_i)^r$ by the multinomial theorem, and thus to 1 for any  $x\in \Delta_n$.

There is a vast literature on Bernstein approximation, and the interested reader may consult e.g.\ the papers by Ditzian \cite{Ditzian1,Ditzian2},
Ditzian and Zhou \cite{Ditzian3}, the book by Altomare and Campiti \cite{Altomare}, and the references therein for more details than given here.

{To motivate our use of Bernstein approximation, we state one well-known result that shows uniform convergence.}
\begin{theorem}[See e.g.\ \cite{Altomare}, \S 5.2.11]
\label{th:Altomare}
Let $f:\mathbb{R}^n \rightarrow \mathbb{R}$ be any continuous function defined on $\Delta_n$, and $B_r(f)$ as defined in (\ref{bernstein-approximation}).
One has
\[
\left|B_r(f)(x) - f(x)\right| \le  2\omega\left(f,\frac{1}{\sqrt r}\right) \quad \quad \forall x \in \Delta_n,
\]
where $\omega$ denotes the modulus of continuity:
\[
\omega(f,\delta) := \max_{\stackrel{x,y \in \Delta_n}{\|x-y\| \le \delta}} |f(x) - f(y)| \quad \quad (\delta \ge 0).
\]
\end{theorem}

Next we state some simple inequalities relating a polynomial, its Bernstein approximation {and their minimum over the set $\Delta(n,r)$ of grid points}.

\begin{lemma}\label{claim1}
Given a polynomial $f\in \mathcal{H}_{n,d}$ and $r\ge 1$, one has
\begin{equation}\label{eqA}
\min_{x\in\Delta_n}B_r(f)(x)\ge f_{\Delta(n,r)},
\end{equation}
\begin{equation}\label{eqB}
f_{\Delta(n,r)}-\underline{f} \le \min_{x\in \Delta(n,r)} B_r(f)(x)-\underline{f} \le
\max_{x\in \Delta_n} \{B_r(f)(x)-f(x)\}.\end{equation}
\end{lemma}

\begin{proof}
We first show (\ref{eqA}).
Fix  $x\in \Delta_n$. By the multinomial theorem,  $1=(\sum_{i=1}^nx_i)^r=\sum_{\alpha\in I(n,r)}{r!\over\alpha!}x^{\alpha}$.
Hence,  $B_r(f)(x)$ is a convex combination of the values
$f({\alpha\over r})$  ($\alpha\in I(n,r)$), which implies that
$B_r(f)(x) \ge \min_{\alpha\in I(n,r)}f({\alpha\over r}) = f_{\Delta(n,r)}.$

 The left most inequality in (\ref{eqB}) follows directly from (\ref{eqA}).
 To show the right most inequality, let $x^*$  be a global minimizer of $f$ over $\Delta_n$, so that
$f(x^*)=\underline{f}$.  Then,
$\min_{x\in \Delta_n}B_r(f)(x) -\underline{f}$ is at most $B_r(f)(x^*)-\underline{f}=B_r(f)(x^*)-f(x^*),$ which concludes the proof.
\qed
\end{proof}
{The motivation for using Bernstein approximation} to study the quantity $f_{\Delta(n,r)}$ is {now} clear from Theorem \ref{th:Altomare} and relation \eqref{eqA}. Indeed, the Bernstein approximation $B_r(f)$ converges uniformly to $f$ as $r \rightarrow \infty$, and the minimum of $B_r(f)$ on $\Delta_n$ is lower bounded
by $f_{\Delta(n,r)}$.

{Our strategy for upper bounding the range $f_{\Delta(n,r)}-\underline{f}$ will be to upper bound the (possibly larger) range
$\max_{x\in \Delta_n} \{B_r(f)(x)-f(x)\}$ -- see Theorems \ref{berineq}, \ref{thmcubic}, \ref{thmsf} and \ref{thmptas2}.
Hence our results can be seen as refinements of the previously known results quoted in Theorem \ref{thm:PTAS} above.}

\medskip
\noindent The following example shows  that all inequalities   can be strict in  relation (\ref{eqB}).

%\begin{remark}\label{remark1}
%The inequality in Lemma \ref{claim1} can be strict as the following  example shows.
%Consider the polynomial  $f=x_1^2+x_2^2\in\mathcal{H}_{2,2}$. One can easily check that $\underline{f}=f_{\Delta(2,2)}={1\over 2}$.
%Moreover, for $x\in\Delta_2$, one has $B_2(f)(x)=1-x_1x_2$ and $\min_{x\in\Delta_{2}}B_2(f)(x)={3\over 4}$. Therefore, $f_{\Delta(2,2)}<\min_{x\in\Delta_{2}}B_2(f)(x)$.
%\end{remark}

\begin{example}\label{ex0}
%The minimizers of $f(x)$ and $B_r(f)(x)$ on simplex may be different.
Consider the quadratic polynomial $f=2x_1^2+x_2^2-5x_1x_2\in \mathcal{H}_{2,2}$. Then,
 $B_2(f)(x)=x_1^2+{1\over 2}x_2^2-{5\over 2}x_1x_2+x_1+{1\over 2}x_2$.
 One can easily check that $\underline{f}= -{17\over 32}$ (attained at the unique minimizer $({7\over 16}, {9\over 16})$),
 $\min_{x\in \Delta_2}B_2(f)(x)= {7\over 16}$ (attained at the unique minimizer $x=({3\over 8},{5\over 8})$), and $f_{\Delta(2,2)}= -{1\over 2}$ (attained at the unique minimizer $({1\over 2}, {1\over 2})$). In this example,  the polynomial $f$ and its Bernstein approximation $B_2(f)(x)$ do not have a common minimizer over the simplex. % By GloptiPoly 3, one can check the minimizer of the program $\min_{x\in\Delta_2}f(x)$ is $x_1=0.4375,x_2=0.5625$ with the objective value equal to $-0.5312$, while $x_1=0.3750,x_2=0.6250$ is the minimizer of $\min_{x\in\Delta_2}B_2(f)(x)$ with the objective value equal to $0.4375$. And these two minimizers give larger objective values of each other, so the minimizers of these two programs are different.

%For further reference let us also note
Moreover, we note  that $\overline{f}=2$  and $\max_{x\in \Delta_2}\{B_2(f)(x) -f(x)\} = 1$, so that we have the following chain of strict inequalities:
\begin{equation*}\label{eqstrict}
f_{\Delta(2,2)}-\underline{f} \ (={1\over 32})  < \min_{x\in \Delta_2} B_2(f)(x) - \underline{f} \ (={31\over 32}) < \max_{x\in \Delta_2}\{B_2(f)(x) -f(x)\} \ (= 1)  < {1\over 2}(\overline{f}-\underline{f})
\ (= {81\over 64}),
\end{equation*}
which shows that all the inequalities can be strict in (\ref{eqB}).

\end{example}

\subsection{Bernstein coefficients}
For any  polynomial $f=\sum_{\beta\in I(n,d)}f_{\beta}x^{\beta}\in\mathcal{H}_{n,d}$, one can  write
$$f=\sum_{\beta\in I(n,d)}f_{\beta}x^{\beta}=\sum_{\beta\in I(n,d)}\left(f_{\beta}{\beta!\over d!}\right){d!\over \beta!}x^{\beta}.$$
We call $f_{\beta}{\beta!\over d!}$ ($\beta\in I(n,d)$)  the {\em Bernstein coefficients} of $f$, since they are the coefficients of the polynomial $f$ when it is expressed in the Bernstein basis $\{{d!\over \beta!}x^\beta: \beta\in I(n,d)\}$ of $\mathcal{H}_{n,d}$.
 Using the multinomial theorem (as in the proof of Lemma \ref{claim1}), one can see  that, for $x\in \Delta_n$, $f(x)$ is a convex combination of
its Bernstein coefficients $f_{\beta}{\beta!\over d!}$ ($\beta\in I(n,d)$). Therefore, one has
\begin{equation}\label{reprop1}
\min_{\beta\in I(n,d)}f_{\beta}{\beta!\over d!}\le\underline{f}\le\overline{f}\le\max_{\beta\in I(n,d)}f_{\beta}{\beta!\over d!}.
\end{equation}
We will use the following result of \cite{KLP06}, which bounds the range of the Bernstein coefficients in terms of the range of function values.

\begin{theorem}\label{thmptas1}\cite[Theorem 2.2]{KLP06}
For any polynomial $f=\sum_{\beta\in I(n,d)}f_{\beta}x^{\beta}\in\mathcal{H}_{n,d}$ and $x\in\Delta_n$, one has
\begin{equation*}
\overline{f}-\underline{f}\le \max_{\beta\in I(n,d)}f_{\beta}{\beta!\over d!}-\min_{\beta\in I(n,d)}f_{\beta}{\beta!\over d!}\le {2d-1\choose d}d^d (\overline{f}-\underline{f}).
\end{equation*}
\end{theorem}

\subsection{Structure of the paper}\label{subsec1.2}
The paper is organized as follows. In Section \ref{sec2}, we give an elementary proof of the PTAS for
 quadratic polynomial optimization over the simplex, that is closely related to a proof given by Nesterov \cite{Nes2003}.
  In Section \ref{sec3}, we refine the known PTAS result for  cubic polynomial optimization over the simplex.
  In addition, we give an elementary proof of the PTAS result for square-free polynomial optimization over the simplex in Section \ref{sec4}.
   Moreover, in Section \ref{sec5}, we provide an alternative proof of the PTAS for general (fixed-degree) polynomial optimization over the simplex.
   We conclude with a discussion of the exact relation between our analysis and that by Nesterov \cite{Nes2003} in Section \ref{sec:conclusion}.
In the {Appendix} we provide a self-contained proof for an explicit description of the moments of the multinomial distribution in terms of the Stirling numbers of the second kind.

\subsection{Notation}\label{subsec1.3}
Throughout we use the notation $[n]=\{1,2,\ldots,n\}$ and $\oN^n$ is the set of all nonnegative integral vectors. For $\alpha\in\oN^n$, we define $|\alpha|=\sum_{i=1}^n\alpha_i$ and  $\alpha!=\alpha_1!\alpha_2!\cdots\alpha_n!$. % and denote its support as $I(\alpha)$.
For two vectors $\alpha,\beta\in\oN^n$, the inequality $\alpha\le\beta$ is coordinate-wise and means that  $\alpha_i\le\beta_i$ for any $i=1,\dots,n$. We  let $e_1,\dots,e_n$ denote the standard unit vectors in $\oR^n$. Moreover, for $I\subseteq [n]$ we set $e_I=\sum_{i\in I} e_i$ and we let $e$ denote the all-ones vector in $\oR^n$.
As before, $I(n,d)=\{\alpha\in \oN^n: |\alpha|=d\}$ and $\mathcal{H}_{n,d}$ denotes the set of all multivariate homogeneous polynomials in $n$ variables with degree $d$. Monomials in $\mathcal{H}_{n,d}$ are denoted as $x^{\alpha}=\prod_{i=1}^nx_i^{\alpha_i}$ for $\alpha\in I(n,d)$, while for $I\subseteq[n]$, we use the notation $x^{I}=\prod_{i\in I}x_i$.
% Given a matrix $Q\in\oR^{n\times n}$, let $Q_{\max}:=\max_{i}Q_{ii}$. For nonnegative integers $n$ and $k$, we denote the Stirling numbers of the second kind as $S(n,k)$; recall that $S(n,k)$ counts the ways to divide a set of $n$ objects into $k$ nonempty subsets.
Finally, for $\beta\in \oN^n$, we also use $\phi_\beta$ to denote the monomial $x^\beta$, i.e., we set $\phi_\beta(x)=x^\beta$.

\section{PTAS for  quadratic polynomial optimization over the simplex}\label{sec2}

\noindent %Let $x\in\Delta_n$, denote
We first recall the explicit Bernstein approximation of the monomials  of degree at most two, i.e., we compute $B_r(\phi_{e_i}),$ $ B_r(\phi_{2e_i})$ and $B_r(\phi_{e_i+e_j})$. %$\phi_i^{(1)}(x)=x_i$, $\phi_i^{(2)}(x)=x_i^2$ and $\phi_{ij}^{(1,1)}(x)=x_ix_j$ for $i,j\in [n]$.
We give a proof for clarity.

%Then we have the following result.
\begin{lemma}\label{propbernstein}
For $r\ge 1$ one has  $B_r(\phi_{e_i})(x)=x_i, B_r(\phi_{2e_i})(x)={1\over r}x_i(1-x_i)+x_i^2$, and $B_r(\phi_{e_i+e_j})(x)={r-1\over r}x_ix_j$ for all $x\in \Delta_n$.
\end{lemma}

\begin{proof}
By the definition (\ref{bernstein-approximation}), one has:
\begin{eqnarray*}
B_r(\phi_{e_i})(x)&=&\sum_{\alpha\in I(n,r)}{\alpha_i\over r}{r!\over \alpha!}x^{\alpha}= x_i \sum_{\beta\in I(n,r-1)}{(r-1)!\over \beta!}x^{\beta}
= x_i (\sum_{i=1}^n x_i)^{r-1}
= x_i,\\
%\end{eqnarray*}
%\begin{eqnarray*}
B_r(\phi_{2e_i})(x)&=&\sum_{\alpha\in I(n,r)}{\alpha_i^2\over r^2}{r!\over \alpha!}x^{\alpha}
= {r-1\over r}x_i^2 \sum_{\beta\in I(n,r-2)}{(r-2)!\over \beta!}x^{\beta}+{1\over r}x_i \sum_{\beta\in I(n,r-1)}{(r-1)!\over \beta!}x^{\beta} \\
&=& {r-1\over r} x_i^2 +{1\over r}x_i = {1\over r}x_i(1-x_i)+x_i^2,\\
%\end{eqnarray*}
%\begin{eqnarray*}
B_r(\phi_{e_i+e_j})(x)&=&\sum_{\alpha\in I(n,r)}{\alpha_i\alpha_j\over r^2}{r!\over \alpha!}x^{\alpha}
= {r-1\over r}x_ix_j \sum_{\beta\in I(n,r-2)}{(r-2)!\over \beta!}x^{\beta}
= {r-1\over r}x_ix_j,
\end{eqnarray*}
where we have  used at several places the multinomial theorem (and the fact that an empty summation is equal to 0).
\qed
\end{proof}

\noindent
Consider now a quadratic polynomial $f=x^TQx\in \mathcal{H}_{n,2}$.  By Lemma \ref{propbernstein}, its Bernstein approximation on the simplex is given by
\begin{equation}\label{brf}
B_r(f)(x)={1\over r}\sum_{i=1}^nQ_{ii}x_i+(1-{1\over r})f(x) \ \ \forall x\in \Delta_n.
\end{equation}

\begin{theorem}\label{berineq}
For any polynomial $f=x^TQx\in \mathcal{H}_{n,2}$ and $r\ge 1$, one has
%$$ \tcolred{\min_{x\in\Delta_n}B_r(f)(x)-\underline{f}\le
$${\max_{x\in \Delta_n}\{B_r(f)(x) -f(x)\}\le}
%\min_{x\in\Delta_n}B_r(f)(x)-\underline{f}\le
{Q_{\max}-\underline{f}\over r} \le
{\overline{f}-\underline{f}\over r}.$$ %\ \ \ \forall x\in \Delta_n,$$
setting $Q_{\max}=\max_{i\in [n]} Q_{ii}$.
\end{theorem}

\begin{proof}
%First we show the left most inequality. For this, choose $x^*$ to be a global minimizer of $f$ over $\Delta_n$, so that
%$f(x^*)=\underline{f}$, combined with    the fact that
%$\min_{x\in \Delta_n}B_r(f)(x) -\underline{f}\le B_r(f)(x^*)-\underline{f}=B_r(f)(x^*)-f(x^*).$
%
%We now show the right most inequality.
Using (\ref{brf}), one obtains that
\begin{eqnarray*}
rB_r(f)(x)&=&\sum_{i=1}^nQ_{ii}x_i+(r-1)f(x)\\
&\le & \max_{x\in\Delta_n}\sum_{i=1}^nQ_{ii}x_i+rf(x)-\min_{x\in\Delta_n}f(x)\\
&=& \max_{i}Q_{ii}-\underline{f}+rf(x)\\
&\le & \overline{f}-\underline{f}+rf(x),
\end{eqnarray*}
where in the last inequality we have used the fact that $\max_i Q_{ii} \le \overline{f}$, since $Q_{ii}=f(e_i)\le \overline{f}$ for $i\in [n]$.
This gives the  two right-most inequalities in the theorem.
\qed
\end{proof}

\noindent Combining  Theorem \ref{berineq}  with Lemma \ref{claim1}, we obtain the following corollary, which gives the PTAS result by Bomze and de Klerk \cite[Theorem 3.2]{BK02}.

\begin{corollary}\label{corberineq}
For any polynomial $f=x^TQx\in \mathcal{H}_{n,2}$ and $r\ge 1$, one has $$f_{\Delta(n,r)}-\underline{f}\le {Q_{\max}-\underline{f}\over r} \le {\overline{f}-\underline{f}\over r}.$$
\end{corollary}
We note that the proof given here is completely elementary and much simpler than the original one in \cite{BK02}.
Our proof is, however, closely related to
another proof by Nesterov \cite{Nes2003}, we will give the precise relation in Section \ref{sec:conclusion}.

\begin{example}\label{exquad}
Consider the quadratic polynomial $f=\sum_{i=1}^n x_i^2 \in \mathcal{H}_{n,2}$.
As $f$ is convex, it is easy to check that $\underline{f}={1\over n}$ (attained at $x={1\over n}e$) and $\overline{f}=1$ (attained at any standard unit vector).

For the computation of $f_{\Delta(n,r)}$, % we distinguish two cases depending whether $n\ge r$ or $n<r$. Namely,
it is convenient to write $r$ as $r=kn+s$, where $k\ge 0$ and $0\le s<n$. Then we have
%\begin{displaymath}
%f_{\Delta(n,r)}=\left\{
%\begin{array}{ll}
%{1\over r}  & \text{if } r\le n, \\
%{1\over n}+{1\over r^2} {s(n-s)\over n} & \text{if } r\ge n, \text{ where } r=kn+s \text{ with } 0\le s<n \text{ and }  k\ge 1.
%\end{array} \right.
%\end{displaymath}
$$f_{\Delta(n,r)}={1\over n} + {1\over r^2}{s(n-s)\over n},$$
which is attained at
any point $x\in \Delta(n,r)$ having $n-s$ coordinates equal to $k\over r$ and $s$ coordinates equal to $k+1\over r$.
To see this, pick a minimizer $x\in \Delta(n,r)$. First we claim that $x_i-x_j\le {1\over r}$ for any $i\ne j\in [n]$. Indeed, assume (say) that $x_2-x_1>{1\over r}$.
Then define the new point $x'\in \Delta(n,r)$  by $x'_1=x_1+{1\over r}$, $x_2'=x_2 -{1\over r}$ and $x_i'=x_i$ for all $i\ne 1,2$
and observe that  $f(x')< f(x)$, which contradicts the optimality of $x$. Therefore, the coordinates of $x$ can  take at most two possible values ${h\over r}, {h+1\over r}$ for some $0\le h\le r-1$ and it is easy to see these two values belong to $\{{k\over r},{k+1\over r}\}$.
Hence we obtain that
$$f_{\Delta(n,r)}-\underline{f}={1\over r^2}{s(n-s)\over n} \ \ \text{ and } \
{f_{\Delta(n,r)}-\underline f\over \overline f-\underline f} = {1\over r^2}{s(n-s)\over n-1}.$$
%Finally observe that the range $f_{\Delta(n,r)}-\underline{f}$ may grow as   $1\over r$ for certain values of $r$.
We observe that this latter ratio might be in the order $1\over r$, thus matching the upper bound in Corollary \ref{corberineq} in
terms of the dependence of the error bound on $r$.
 For instance, for $r={3n\over 2}$ (i.e., $k=1$, $s={n\over 2}$), we have that
\begin{equation}\label{eqratior}
{f_{\Delta(n,r)}-\underline{f}}= {1\over 6r-9}( \overline f -\underline f).
\end{equation}
%\tcolred{Question: In order to say that the range ${f_{\Delta(n,r)}-\underline{f}\over \overline f-\underline f}$ may grow in the order $1\over r$, I guess we need to exhibit an infinite  sequence of values of $r$ for which this is the case. So far, once $f$ is given (i.e., $n$ is given), we have just one value $r=3n/2$...}
Moreover, we have $B_r(f)(x)= {1\over r} +(1-{1\over r})f(x)$, so that
$$\min_{x\in \Delta_n} B_r(f)(x) -\underline{f}=\max_{x\in \Delta_n} \{B_r(f)(x) -f(x)\} = {1\over r}(\overline{f}-\underline{f}).$$
Hence, equality holds throughout  in the inequalities of Theorem \ref{berineq}, which shows that the upper bound is tight on this example.
%On the other hand, on this example, we see that  the upper bound grows as $1\over r$ while the range $f_{\Delta(n,r)}-\underline{f}$ grows as $1\over r^2$. It turns out that this will be the case also for the other polynomials considered in Examples \ref{excubic} and \ref{exsf}.

\end{example}

\noindent By Example \ref{exquad}, there does not exist any $\epsilon >0$ such that, for any quadratic form $f$,
\[
f_{\Delta(n,r)}-\underline{f} \le \frac{1}{r^{1+\epsilon}}\left(\overline{f} - \underline{f}\right) \quad  \forall r \ge 1,
\]
and thus the error bound in Corollary \ref{corberineq} is tight  in terms of its dependence on $r$.
On the other hand,
one may easily show that, for the polynomial $f$ in Example \ref{exquad},
\[
\rho_r(f) := \frac{f_{\Delta(n,r)}-\underline{f}}{\overline{f} - \underline{f}} \le \frac{n}{4r^2} = O(1/r^2).
\]
Thus, $\limsup_{r\rightarrow \infty} (r^2\rho_r(f)) < \infty$, i.e.\ the asymptotic convergence rate of the sequence $\{\rho_r(f)\}$ for the
 example is {$O(1/r^2)$}. It turns out that this will be the case also for the other polynomials considered in Examples \ref{excubic}, \ref{exsf} and \ref{exdegd}.

\section{PTAS for  cubic polynomial optimization over the simplex}\label{sec3}

\noindent Using similar arguments as for Lemma \ref{propbernstein},  one can compute the Bernstein approximations of the monomials of degree three. Namely, for distinct $i,j,k\in [n]$ and $x\in \Delta_n$,
%$\phi_i^{(3)}(x)=x_i^3$, $\phi_{ij}^{(2,1)}(x)=x_i^2x_j$ and $\phi_{ijk}^{(1,1,1)}(x)=x_ix_jx_k$ read
\begin{eqnarray*}
B_r(\phi_{3e_i})(x)&=&{1\over r^2}x_i+{3(r-1)\over r^2}x_i^2+{(r-1)(r-2)\over r^2}x_i^3,\\
B_r(\phi_{2e_i+e_j})(x)&=&{(r-1)\over r^2}x_ix_j+{(r-1)(r-2)\over r^2}x_i^2x_j,\\
B_r(\phi_{e_i+e_j+e_k})(x)&=&{(r-1)(r-2)\over r^2}x_ix_j x_k.
\end{eqnarray*}

\noindent %Similarly as in Theorem \ref{berineq}, we can
We show the following result.
\begin{theorem}\label{thmcubic}
For any polynomial $f\in \mathcal{H}_{n,3}$ and $r\ge 2$, one has
$$ %\tcolred{\min_{x\in\Delta_n}B_r(f)(x)-\underline{f}\le
{ \max_{x\in \Delta_n}\{B_r(f)(x)- f(x)\}\le }
 \left({4\over r}-{4\over r^2}\right)(\overline{f}-\underline{f}).$$ % \ \ \ \forall x\in \Delta_n.$$
\end{theorem}

\begin{proof}
Consider a cubic  polynomial $f\in \mathcal{H}_{n,3}$ of the form
\begin{equation*}\label{re11}
f=\sum_{i=1}^nf_ix_i^3+\sum_{1\le i<j\le n}(f_{ij}x_ix_j^2+g_{ij}x_i^2x_j)+\sum_{1\le i<j<k\le n}f_{ijk}x_ix_jx_k.
\end{equation*}
Applying the above  description for the Bernstein {approximation} of degree 3 monomials, the Bernstein approximation of $f$ at any  $x\in\Delta_n$ reads
\begin{equation}\label{berneqcubic}
B_r(f)(x)={(r-1)(r-2)\over r^2}f(x)+{1\over r^2}\left[\sum_{i=1}^nf_ix_i+(r-1)\left(\sum_{i=1}^n3f_ix_i^2+\sum_{i<j}(f_{ij}+g_{ij})x_ix_j\right)\right].
\end{equation}
Evaluating $f$ at $e_i$ and at $(e_i+e_j)/2$ yields, respectively, the relations:
\begin{eqnarray}
\underline{f}\le f_i\le\overline{f},\label{re1}\\
f_i+f_j+f_{ij}+g_{ij}\le 8\overline{f}.\label{re2}
\end{eqnarray}
Using (\ref{re2}) and the fact that $\sum_{i=1}^nx_i=1$, one can obtain
\begin{equation}\label{re3}
\sum_{i<j}(f_{ij}+g_{ij})x_ix_j\le\sum_{i<j}(8\overline{f}-f_i-f_j)x_ix_j=8\overline{f}\sum_{i<j}x_ix_j-\sum_{i=1}^n
f_ix_i(1-x_i).
\end{equation}
Combining (\ref{berneqcubic}) and (\ref{re3}), one obtains that, for any $x\in \Delta_n$,
\begin{eqnarray*}
 %r^2\min_{x\in\Delta_n}
r^2 B_r(f)(x)
%&=& %\min_{x\in\Delta_n}
%\left[
& = & (r-1)(r-2)f(x)+\sum_{i=1}^nf_ix_i+(r-1)\left(\sum_{i=1}^n3f_ix_i^2+\sum_{i<j}(f_{ij}+g_{ij})x_ix_j\right)\\
%\right]\\
&\le&%\min_{x\in\Delta_n}
%\left[
(r-1)(r-2)f(x)-(r-2)\sum_{i=1}^nf_ix_i+(r-1)\left(\sum_{i=1}^n4f_ix_i^2+8\overline{f}\sum_{i<j}x_ix_j\right).
%\right].
\end{eqnarray*}
We now use (\ref{re1}) to bound the two inner summations as follows:
\begin{eqnarray*}
-\sum_i f_ix_i\le -\underline{f}\sum_ix_i=-\underline{f}\ \text{ and } \
\sum_{i=1}^n4f_ix_i^2+8\overline{f}\sum_{i<j}x_ix_j \le 4 \overline{f} (\sum_ix_i)^2=4\overline{f}.
\end{eqnarray*}
This implies:
{
\begin{eqnarray*}
r^2 (B_r(f)(x)-f(x))&\le&
%\min_{x\in\Delta_n}\left[(r-1)(r-2)f(x)-(r-2)\sum_{i=1}^nf_ix_i+4(r-1)\overline{f}\right]\\
-(3r-2)\underline{f}  -(r-2)\underline{f}
  + 4(r-1)\overline{f}  = 4(r-1)(\overline{f}-\underline{f}),
\end{eqnarray*}
}
which concludes the proof.
\qed
\end{proof}

\noindent  Combining Theorem \ref{thmcubic} with Lemma \ref{claim1}, we obtain the following error bound.
\begin{corollary}
For any polynomial $f\in \mathcal{H}_{n,3}$ and $r\ge 2$, one has
\begin{equation*}\label{ptascube}
f_{\Delta(n,r)}-\underline{f}\le \left({4\over r}-{4\over r^2}\right)(\overline{f}-\underline{f}).
\end{equation*}
\end{corollary}

\noindent
This result is a bit stronger than the result by de Klerk et al. \cite[Theorem 3.3]{KLP06}, which states that $f_{\Delta(n,r)}-\underline{f}\le {4\over r}(\overline{f}-\underline{f})$.

\begin{example}\label{excubic}
Consider the cubic polynomial $f=x_1^3+x_2^3\in \mathcal{H}_{2,3}$. One can check that
$\overline{f}=1,\underline{f}={1\over 4}$,
\begin{displaymath}
f_{\Delta(2,r)}=\left\{ \begin{array}{ll}
1/4 & \text{if $r$ is even,}\\
{1\over 4}+{3\over 4r^2} & \text{if $r$ is odd.}
\end{array} \right.
\end{displaymath}
Moreover, one can check that
$B_r(f)(x)=1+\left({3\over r}-3\right)x_1x_2$ and
$\min_{x\in\Delta_2}B_r(f)(x)={1\over 4}+{3\over 4r}.$
Hence,  for any integer $r\ge 2$, one has strict inequality $\min_{x\in\Delta_2}B_r(f)(x)>f_{\Delta(2,r)}$. Moreover, for $r\ge 2$,
$$\min_{x\in \Delta_2} B_r(f)(x)-\underline{f} = \max_{x\in \Delta_2}\{B_r(f)(x)-f(x)\}= {3\over 4r} = {1\over r}(\overline{f}-\underline{f}) < ({4\over r}-{4\over r^2})(\overline{f}-\underline{f}).$$
 %so that equality holds throughout in Theorem \ref{thmcubic} on this example.
 On the other hand, for odd $r$,  the range  $f_{\Delta(2,r)}-\underline{f}$ is equal to $ {3\over 4r^2}={1\over r^2}(\overline{f}-\underline{f})$ and thus grows proportionally to  $1\over r^2$.
\end{example}

\section{PTAS for square-free polynomial optimization over the simplex}\label{sec4}

\noindent %Similarly as in Proposition \ref{propbernstein}, if $x\in\Delta_n$, for
Here we consider square-free (aka multilinear) polynomials, involving only monomials $x^I$ for $I\subseteq [n]$.
The Bernstein approximation of
 the square-free monomial $\phi_{e_I}(x):=x^I$, with $d=|I|$,  is given by
\begin{eqnarray*}\label{propbernstein2}
B_r(\phi_{e_I})(x)=\sum_{\alpha\in I(n,r)}{\alpha^I\over r^d}{r!\over \alpha!}x^{\alpha}
= {r^{\underline{d}}\over r^{d}}x^I \sum_{\alpha\in I(n,r-d)}{(r-d)!\over \alpha!}x^{\alpha}
=  {r^{\underline{d}}\over r^{d}}x^I(\sum_ix_i)^{r-d}=
 {r^{\underline{d}}\over r^{d}}x^I \ \  %\forall x\in \Delta_n.
\end{eqnarray*}
for  $x\in \Delta_n$. Recall that, for an integer $r\ge 1$,  $r^{\underline d}=r(r-1)\cdots (r-d+1)$ and observe that $r^{\underline d}=0$ if $r<d$.
\noindent Hence the Bernstein approximation of  the square-free polynomial $f=\sum_{I\subseteq [n],|I|=d}f_Ix^I$ satisfies
 \begin{equation*}\label{eqsf}
B_r(f)(x)={r^{\underline{d}}\over r^{d}}f(x)  \ \ \ \forall x\in \Delta_n,
\end{equation*}
which implies the following identities:
\begin{equation}\label{eqsf1}
{\min_{x\in \Delta_n} B_r(f)(x) -\underline {f} = \max_{x\in \Delta_n} \{B_r(f)(x)-f(x)\} =-\left(1-{r^{\underline d}\over r^d}\right)\underline{f}. }
\end{equation}

\begin{theorem}\label{thmsf}
For any square-free polynomial $f\in \mathcal{H}_{n,d}$ and $r\ge 1$, one has
\begin{equation*}
%f_{\Delta(n,r)}-\underline{f}\le
%\min_{x\in\Delta_n}B_r(f)(x)-\underline{f}\le
%\min_{x\in \Delta_n} B_r(f)(x) -\underline {f} =
{ \max_{x\in \Delta_n} \{B_r(f)(x)-f(x)\} }\le  \left(1-{r^{\underline{d}}\over r^{d}}\right)(\overline{f}-\underline{f})\ \ \ \forall x\in \Delta_n.
\end{equation*}
\end{theorem}

\begin{proof}
 %Lemma \ref{claim1} gives the left most inequality and the  right most inequality follows using (\ref{eqsf}).
%Combining  (\ref{eqsf}) with Lemma \ref{claim1}, we obtain
%\begin{equation*}
%f_{\Delta(n,r)}\le \min_{x\in \Delta_n} B_r(f)(x) = {r^{\underline{d}}\over r^{d}}\underline{f}.
%\end{equation*}
%Moreover, we have
We use (\ref{eqsf1}). For degree $d=1$ the result is clear and, for $d\ge 2$,  we use the fact that $\overline{f}\ge 0$ since $f(e_i)=0$ for any $i\in [n]$.\qed
\end{proof}

\noindent Combining with Lemma \ref{claim1}  we obtain the following error bound.
\begin{corollary}
For any square-free polynomial $f\in \mathcal{H}_{n,d}$ and $r\ge 1$, one has
\[
\label{square free PTAS}
f_{\Delta(n,r)}-\underline{f}\le \left(1-{r^{\underline{d}}\over r^{d}}\right)(\overline{f}-\underline{f}).
\]
\end{corollary}
\noindent
This  result was first shown by Nesterov \cite[Theorem 2]{Nes2003} (see also De Klerk, Laurent, and Parrilo \cite[Remark 3.4]{KLP06}).
In fact, our proof is again closely related to the one by Nesterov; see Section \ref{sec:conclusion} for the details.

\begin{example}\label{exsf}
Consider the square-free polynomial $f=-x_1x_2$. Then,  $B_r(f)(x)=-{r-1\over r}x_1x_2$ and
one can check that
$\overline{f}=0,$ $\underline{f}=-{1\over 4}$, and $ \min_{x\in\Delta_2}B_r(f)(x)=-{1\over 4}{r-1\over r}.$ Moreover,
\begin{displaymath}
f_{\Delta(2,r)}=\left\{ \begin{array}{ll}
-{1\over 4} & \text{if $r$ is even,}\\
-{1\over 4}+ {1\over 4r^2} & \text{if $r$ is odd.}
\end{array} \right.
\end{displaymath}
Hence, for any integer $r\ge 2$, one has strict inequality: $\min_{x\in\Delta_2}B_r(f)(x)>f_{\Delta(n,r)}$.
Moreover, as $\min_{x\in \Delta_2} B_r(f)(x)-\underline{f} = \max_{x\in \Delta_2}\{B_r(f)(x)-f(x)\}= {1\over 4r} = {1\over r}(\overline{f}-\underline{f}),$  the upper bound from Theorem \ref{thmsf} is tight on this example. On the other hand,   $f_{\Delta(2,r)}-\underline{f}= {1\over 4r^2}= {1\over r^2}(\overline{f}-\underline{f})$ for odd $r$, and thus the range  $f_{\Delta(2,r)}-\underline{f}$ grows proportionally to $1\over r^2$.
\end{example}

%, who showed that
%\begin{equation}
%\label{Nesterov square free bound}
%f_{\Delta(n,r)} - \underline{f} \le \left(1- {r^{\underline{d}}\over r^{d}}\right)(-\underline{f}),
%\end{equation}
%which clearly implies \eqref{square free PTAS} if $\overline{f} > 0$.
% However, the proof given here is again elementary, as opposed to the one in
%\cite{Nes2003} where properties of random walks on a simplex are used.
% Moreover, the bound \eqref{square free PTAS} is stronger than
%\eqref{Nesterov square free bound} if $\overline{f} < 0$. (The assumption $\overline{f} > 0$ cannot be made without loss of generality in the square-free case).

\section{PTAS for  general polynomial optimization over the simplex}\label{sec5}

\noindent
In this section we deal with the minimization of an arbitrary polynomial $f\in \mathcal{H}_{n,d}$. In order to be able to bound  the minimum of $B_r(f)$ over $\Delta_n$ %\tcolblue{(should be $\Delta_n$?)} %$\min_{x\in \Delta(n,d)}B_r(f)(x)$
we need an explicit description of the Bernstein approximation of $f$.

\subsection{Bernstein approximation over the simplex of an arbitrary monomial}
Here we work out an explicit description of the  Bernstein approximation
 of  arbitrary monomials
$\phi_\beta(x)= x^\beta$ ($\beta\in I(n,d)$).
The key ingredient is to express it in terms of  the moments of the multinomial distribution.

Fix $x=(x_1,\ldots,x_n)\in \Delta_n$ and  consider the multinomial distribution with $n$ categories and $r$ independent trials,
 where the probability for the $i$-th category is given by $x_i$. Then, given $\alpha\in I(n,r)$, the probability of drawing $\alpha_i$ times the $i$-th category for each $i\in [n]$  is equal to ${r!\over \alpha!} x^\alpha$. Therefore, for $\beta\in \oN^n$,  the $\beta$-th moment of this multinomial distribution is given by
\begin{equation*}\label{relmom}
m ^{\beta}_{(n,r)}:=\sum_{\alpha\in I(n,r)}{\alpha}^{\beta}{r!\over \alpha!}x^{\alpha}.
\end{equation*}
%=r^{|\beta|}\sum_{\alpha\in I(n,r)}\left({\alpha\over r}\right)^{\beta}{r!\over \alpha!}x^{\alpha}=r^{|\beta|}B_r(\phi_{I}^{\beta})(x).$$
%where $\phi_{I}^{\beta}(x)=x^{\beta}$, $I=I(\beta)$ and $x\in\Delta_n$.
Comparing with the definition of the Bernstein approximation of $\phi_\beta(x)=x^\beta$ we find the identity
 \begin{equation*}\label{eqmom}
 B_r(\phi_\beta)(x)=\sum_{\alpha\in I(n,r)}({\alpha\over r})^\beta {r!\over \alpha!}x^\beta = {1\over r^{|\beta|}} m ^{\beta}_{(n,r)}.
 \end{equation*}
  Combining \cite[relation (34.18)]{JKB97} and \cite[relation (35.5)]{JKB97}, we can obtain an  explicit formula for the moments  $m ^{\beta}_{(n,r)}$ of the multinomial distribution in terms of the Stirling numbers of the second kind.
  Recall that, for integers $n,k\in \oN$, the {\em Stirling number of the second kind} $S(n,k)$ counts the number of ways of partitioning a set of $n$ objects into $k$ nonempty subsets.
  Thus $S(n,k)=0$ if $k>n$, $S(n,0)=0$ if $n\ge 1$, and $S(0,0)=1$ by convention.

\begin{theorem}\label{theomom}
For $\beta\in \oN^n$, one has
$$m ^{\beta}_{(n,r)}=\sum_{\alpha\in \oN^n:
 \alpha\le\beta}r^{\underline{|\alpha|}}x^{\alpha}\prod_{i=1}^nS(\beta_i,\alpha_i),$$
where $S(\beta_i,\alpha_i)$ are Stirling numbers of the second kind.
\end{theorem}

\noindent Therefore, we can deduce the explicit formula of the Bernstein approximation for any monomial.

\begin{corollary}\label{thmbernstein}
For any monomial $\phi_\beta(x)=x^{\beta}$,  one has
$$B_r(\phi_\beta)(x)={1\over r^{|\beta|}}\sum_{\alpha\in \oN^n: \alpha\le\beta}r^{\underline{|\alpha|}}x^{\alpha}\prod_{i=1}^nS(\beta_i,\alpha_i)\ \ \forall x\in \Delta_n.$$
\end{corollary}
\noindent For completeness, we will give a self-contained proof for Corollary \ref{thmbernstein} in the Appendix.

\subsection{Error bound analysis}

We show the following  error bound for the Bernstein approximation of order $r$ of an arbitrary polynomial  on the simplex.

\begin{theorem}\label{thmptas2}
For any polynomial $f\in\mathcal{H}_{n,d}$ and  $r\ge 1$,  one has
\begin{equation*}\label{reptas4}
\begin{array}{ll}
{
%\displaystyle\min_{x\in \Delta_n} B_r(f)(x)-\underline{f}\le
\max_{x\in \Delta_n}\{B_r(f)(x)-f(x)\}}
& \le {\left(1-{r^{\underline{d}}\over r^d}\right)}  \left(\displaystyle \max_{\beta\in I(n,d)} f_\beta {\beta!\over d!}  -\displaystyle\min_{\beta\in I(n,d)} f_\beta {\beta!\over d!}\right)\\
&\le \left(1-{r^{\underline{d}}\over r^d}\right){2d-1\choose d}d^d(\overline{f}-\underline{f}). % \ \ \forall x\in \Delta_n.
\end{array}
\end{equation*}
\end{theorem}

\noindent For the proof we will need  two auxiliary results about the Stirling numbers of the second kind.
The first result  is implied  by \cite[relation (3.2)]{AK08}, and  we therefore only sketch the proof.

\begin{lemma}\label{lemptas1}
For positive integers $d$ and $r\ge 1$, one has
\begin{equation*}\label{reptas2}
\sum_{k=1}^{d-1}r^{\underline{k}}S(d,k)=r^d-r^{\underline{d}}.
\end{equation*}
\end{lemma}

\begin{proof}
The proof is by induction on $d$, and using relation (\ref{eq3}) (in the appendix to this paper) for the induction step.\qed
\end{proof}
%Moreover, the following lemma --- which gives a alternative expression for Stirling numbers of the second kind --- is also necessary for our proof of Theorem \ref{thmptas2}.
The second result gives an alternative expression for Stirling numbers of the second kind. We provide a full proof, since we could not find this result in
the literature.

\begin{lemma}\label{lemptas2}
Given $\alpha\in I(n,k)$ and $d>k$, one has
\begin{equation}\label{reptas3}
S(d,k)={\alpha!\over k!}\sum_{\beta\in I(n,d)}{d!\over\beta!}\prod_{i=1}^n S(\beta_i,\alpha_i).
\end{equation}
\end{lemma}

\begin{proof}
{
For integers $d,k\ge 0$, let $S_{d,k}$ denote the number of surjective maps from a $d$-elements set to a $k$-elements set.
It is not difficult to see the following relation between $S_{d,k}$ and $S(d,k)$:
\begin{equation*}\label{eqsurj}
S_{d,k}= k! S(d,k).
\end{equation*}
Indeed, let $B=[d]$ and $A=[k]$. In order to choose a surjective map $f$ from $B$ to $A$ one needs to select the pre-image $B_i=f^{-1}(i)\subseteq B$ for each element $i\in [k]$. So to define a surjective map $f$, one first selects a partition of $B$ into $k$ non-empty subsets $B_1,\ldots,B_k$, which can be done in $S(d,k)$ ways. As any permutation of the $B_i's$ gives rise to a distinct surjective map,
 there are $k! S(d,k)$ surjective maps from $[d]$ to $[k]$.

Now, the identity (\ref{reptas3}) about the Stirling numbers $S(d,k)$ can be equivalently reformulated as the following identity about the numbers $S_{d,k}$: For any $\alpha \in I(n,k)$,
\begin{equation*}\label{idsurj}
S_{d,k}=\sum_{\beta \in I(n,d)} {d!\over \beta!} \prod_{i=1}^n S_{\beta_i,\alpha_i}.
\end{equation*}
Again set $B=[d]$ and $A=[k]$.
Say, $\alpha$ has $p$ non-zero coordinates, i.e., $\alpha_1,\ldots,\alpha_p\ge 1$ and $\alpha_1+\ldots+\alpha_p=k$.
Fix a partition of $A=[k]$ into $p$ subsets $A_1,\ldots,A_p$ where $|A_i|=\alpha_i$ for $i\in [p]$.
Then, a surjection $f$ from $B$ to $A$ defines a surjection from $B_i=f^{-1}(A_i)$ to $A_i$ for each $i\in [p]$.
Setting $\beta_i=|B_i|$, we have  $\beta_1+\ldots +\beta_p= d$ since the $B_i's$ partition $B$.
Hence  one can count the number of surjections from $B$ to $A$ as follows.

First, select $\beta_1,\ldots,\beta_p\ge 1$ such that  $\beta_1+\ldots+\beta_p=d$.
Then split the $d$ elements of $B$ into an ordered sequence of $p$ disjoint subsets $B_1,\ldots,B_p$ where $|B_i|=\beta_i$ for $i\in [p]$; there are ${d!\over \beta!}$ ways of doing so.
Once $B_1,\ldots,B_p$ are selected,  there are $S_{\beta_i,\alpha_i}$ possible surjections from $B_i$ to $A_i$ for each $i\in [p]$ and thus a total of
$\prod_{i=1}^p S_{\beta_i,\alpha_i}$ possibilities. Therefore, we get that the total number of surjections from $B$ to $A$ is equal to
$\sum_{\beta\in I(p,d)} {d!\over \beta!} \prod_{i=1}^p S_{\beta_i,\alpha_i}$, which shows the result.
}
\qed
\end{proof}

\noindent We are now ready to prove Theorem \ref{thmptas2}.

\begin{proof} {\em (of Theorem \ref{thmptas2})}
Consider a polynomial $f=\sum_{\beta\in I(n,d)}f_{\beta}x^{\beta}\in\mathcal{H}_{n,d}$ and $x\in\Delta_n$. Applying  Corollary~\ref{thmbernstein}, we can write  the Bernstein approximation of $f$ at $x\in \Delta_n$ as follows:
\begin{equation*}\label{re7}
B_r(f)(x)={1\over r^{d}}\sum_{\beta\in I(n,d)}f_{\beta}\sum_{\alpha:0\le\alpha\le\beta}r^{\underline{|\alpha|}}x^{\alpha}\prod_{i=1}^nS(\beta_i,\alpha_i).
\end{equation*}
%From (\ref{re7}), we obtain
Therefore,
\begin{eqnarray*}
r^{d}B_r(f)(x)=r^{\underline{d}}f(x)+\sum_{\beta\in I(n,d)}f_{\beta}\sum_{\alpha:0\le\alpha\le\beta,\alpha\ne\beta}r^{\underline{|\alpha|}}x^{\alpha}\prod_{i=1}^nS(\beta_i,\alpha_i),\label{re8}
\end{eqnarray*}
and thus
\begin{equation*}\label{reptas1}
r^{d}(B_r(f)(x)-f(x))=-(r^d-r^{\underline{d}})f(x)+\sum_{\beta\in I(n,d)}f_{\beta}\sum_{\alpha: 0\le\alpha\le\beta,\alpha\ne\beta}r^{\underline{|\alpha|}}x^{\alpha}\prod_{i=1}^nS(\beta_i,\alpha_i).
\end{equation*}
Using  (\ref{reprop1}), we have $f(x)\ge \min_{\beta\in I(n,d)} f_\beta {\beta!\over d!}$ and {$f_\beta{\beta!\over d!}\le \max_{\beta'\in I(n,d)} f_{\beta'} {\beta'!\over d!}$}, which permits to derive the following inequality:
\begin{equation}\label{eq00}
\begin{array}{l}
 r^{d}(B_r(f)(x)-f(x))  \le \\
  -(r^d-r^{\underline{d}})\displaystyle \min_{\beta\in I(n,d)}f_{\beta}{\beta!\over d!} +\max_{\beta\in I(n,d)}f_{\beta}{\beta!\over d!}
 \left(\underbrace{\sum_{\beta\in I(n,d)}{d!\over\beta!}\sum_{\alpha:0\le\alpha\le\beta,\alpha\ne\beta}
 x^\alpha r^{\underline{|\alpha|}}\prod_{i=1}^nS(\beta_i,\alpha_i) }_{\sigma}\right).
\end{array}
 \end{equation}
% If we select $x\in \Delta_n$ to be a global minimizer of $f$ over the simplex, then $f(x)=\underline{f}$ and thus
% $\min_{x\in \Delta_n} B_r(f)(x) -\underline{f}\le B_r(f)(x)-f(x)$. Therefore, i
 It now suffices to upper bound the right handside of the inequality (\ref{eq00})
and to show that the inner  summation $\sigma$ is equal to $r^d-r^{\underline d}$. Indeed,
 \begin{eqnarray*}
 \sigma & = & \sum_{\beta\in I(n,d)}{d!\over\beta!}\sum_{\alpha:0\le\alpha\le\beta,\alpha\ne\beta}
 x^\alpha r^{\underline{|\alpha|}}\prod_{i=1}^nS(\beta_i,\alpha_i) \\
&  = &  \sum_{\alpha\in \oN^n} x^\alpha r^{\underline{|\alpha|}}  \sum_{\beta\in I(n,d): \beta \ge \alpha, \beta \ne \alpha} {d!\over \beta!}\prod_{i=1}^nS(\beta_i,\alpha_i) \\
 & = & \sum_{k=1}^{d-1} \sum_{\alpha\in I(n,k)} x^\alpha r^{\underline{|\alpha|}}
 \left(\sum_{\beta\in I(n,d): \beta \ge \alpha} {d!\over \beta!}\prod_{i=1}^nS(\beta_i,\alpha_i) \right)\\
 & = & \sum_{k=1}^{d-1} r^{\underline{k}} \sum_{\alpha\in I(n,k)} x^\alpha \left({k!\over \alpha!} S(d,k)\right) \hspace*{6cm}  \text{[using Lemma \ref{lemptas2}]}\\
 &=& \sum_{k=1}^{d-1} r^{\underline{k}} S(d,k) (\sum_i x_i)^k
 = \sum_{k=1}^{d-1} r^{\underline{k}} S(d,k) = r^d-r^{\underline{d}}. \hspace*{3.3cm} \text{[using Lemma \ref{lemptas1}]}\\
     \end{eqnarray*}
 Using this identity for the summation $\sigma$  in the inequality (\ref{eq00}) we obtain
 $$r^d\left(\min_{x\in \Delta_n} B_r(f)(x) -\underline{f}\right) \le r^{d}\max_{x\in \Delta_n}\{B_r(f)(x)-f(x)\} \le (r^d-r^{\underline{d}})  \left( \max_{\beta\in I(n,d)} f_\beta {\beta!\over d!}  -\min_{\beta\in I(n,d)} f_\beta {\beta!\over d!}\right).$$
By combining with  Theorem \ref{thmptas1} we obtain the claimed inequalities of Theorem \ref{thmptas2} and this concludes the proof.
\qed
\end{proof}

\noindent Combining Theorem \ref{thmptas2} with Lemma  \ref{claim1}, we obtain the following error bound, which was first shown {in} \cite[Theorem 1.3]{KLP06}.

\begin{corollary}\label{corptas2}
For any polynomial $f\in \mathcal{H}_{n,d}$ and $r\ge 1$, one has
\begin{equation*}
\label{final bounds}
f_{\Delta(n,r)}-\underline{f} \le %\min_{x\in\Delta_n}B_r(f)(x)-\underline{f}\le
  \left(1-{r^{\underline{d}}\over r^d}\right){2d-1\choose d}d^d(\overline{f}-\underline{f}).
\end{equation*}
\end{corollary}

%\noindent This upper bound for  the error $f_{\Delta(n,r)}-\underline{f}$  was first shown in \cite[Theorem 1.3]{KLP06}.
%The  following example shows that this inequality may be strict.  Hence the result of Theorem \ref{thmptas2} %and Corollary \ref{corptas2}
%may be seen as a refinement of  \cite[Theorem 1.3]{KLP06}.

\begin{example}\label{exdegd}
We consider here the problem of minimizing the polynomial $f=\sum_{i=1}^n x_i^d $ ($n\ge 2$) over the simplex for any degree $d\ge 2$, thus extending the case $d=2$ considered in Example \ref{exquad} and the case $d=3,n=2$ considered in Example \ref{excubic}.
As $f$ is convex on $\oR^n_+$ it follows that $\overline{f}=1$ and $\underline{f}={1\over n^{d-1}}$.

We now compute the minimum over the regular grid $\Delta(n,r)$. As in Example \ref{exquad} set $r=kn+s$ where $k,s\in \oN$ with $s\le n-1$.
We show that $f_{\Delta(n,r)}$ is attained at any point $x$ having $s$ components equal to $k+1\over r$ and $n-s$ components equal to $k\over r$, so that
$$f_{\Delta(n,r)}= s\left({k+1\over r}\right)^d + (n-s)\left({k\over r}\right)^d.$$
For this pick a minimizer $x$ of $f$ over $\Delta(n,r)$ and it suffices to show that $x_i-x_j\le {1\over r}$ for all $i,j\in [n]$.
If (say) $x_2-x_1>{1\over r}$ then we claim that
$f(x_1+{1\over r},x_2-{1\over r},x_3,\ldots,x_n)<f(x_1,x_2,x_3,\ldots,x_n)$, which contradicts the minimality assumption on $x$.
One can see the above claim as follows: set $\sigma=1-\sum_{i=3}^nx_i$, consider the function $\phi(t)=t^d+(\sigma-t)^d$ for $t$ satisfying $0\le t< {1\over 2}(\sigma -{1\over r})$, and verify (using elementary calculus) that $\phi(t+{1\over r})<\phi(t)$ for any such $t$.
Therefore, we have
\begin{eqnarray*}
f_{\Delta(n,r)}-\underline{f} &=& (n-s)\left({k\over r}\right)^d +s \left({k+1\over r}\right)^d- {n\over n^d}\\
&=& (n-s) \left(\left({k\over r}\right)^d -{1\over n^d}\right) + s \left( \left({k+1\over r}\right)^d -{1\over n^d}\right)\\
&=& {n-s\over n^d} \left( \left( 1-{s\over r}\right)^d-1\right) + {s\over n^d} \left( \left( 1-{s-n\over r}\right)^d -1\right)\\
&=& {s(n-s)\over n^d} \sum_{i=2}^d {d\choose i} {(n-s)^{i-1}+(-1)^i s^{i-1}\over r^i}.\\
\end{eqnarray*}
Using the fact that $s,n-s\le n$,  for any $r\ge n$ we can bound the above summation as follows:
$$ \sum_{i=2}^d {d\choose i} {(n-s)^{i-1}+(-1)^i s^{i-1}\over r^i}\le  {2\over n}\sum_{i=2}^d {d\choose i} \left({n\over r}\right)^i \le
{2\over n} \left({n\over r}\right)^2 \sum_{i=2}^d {d\choose i} \le
{2\over n} \left({n\over r}\right)^2 2^d=
  {2^{d+1}n \over r^2}.$$
Combining with the bound $s(n-s)\le {n^2\over 4}$, we deduce that
$$f_{\Delta(n,r)}-\underline{f} \le {n^2\over 4} {1\over n^d}{2^{d+1} n\over r^2} = {2^{d-1} \over n^{d-3} r^2}. $$
Therefore,
\begin{equation}\label{eqdegd}
{f_{\Delta(n,r)}-\underline{f} \over \overline{f} -\underline{f}} \le
{2^{d-1} \over n^{d-3} r^2}{n^{d-1}\over n^{d-1}-1}
= {2^{d-1} \over r^2} {n^2\over n^{d-1}-1} \le {2^d \over r^2} \ \ \ \text{ for any } r\ge n\ge 2 \text{ and } d\ge 3.
\end{equation}
Hence we see that for any degree $d\ge 3$ the ratio is in the order $1\over r^2$. Recall that for degree $d=2$ it was observed in Example \ref{exquad} that  it can be in the order $1\over r$ for certain values of $r$ (e.g., for $r={3n\over 2}$).
\end{example}

\section{Concluding remarks}
\label{sec:conclusion}
%\subsection{Bernstein approximation and the randomized algorithm of Nesterov}
Nesterov \cite{Nes2003} proposed an alternative probabilistic argument
for  estimating the quality of the bounds
$f_{\Delta(n,r)}$.
He
considered a random walk on the simplex $\Delta_n$, which generates
a sequence of random points
$x^{(r)}\in \Delta(n,r)$ ($r\ge 1$).
Thus the expected value $E(f(x^{(r)}))$ of the evaluation of a polynomial $f \in \mathcal{H}_{n,d}$
at $x^{(r)}$ satisfies:
\begin{equation*}\label{eqNes}
f_{\Delta(n,r)}\le E(f(x^{(r)})).
\end{equation*}
%Then Nesterov \cite{Nes2003} gave explicit  upper bounds for  the expectation
% $E(f(x^{(r)}))$ for quadratic polynomials and for square-free polynomials.
Nesterov's approach goes as follows. Let $x\in \Delta_n$
%denote a global minimizer of the polynomial $f$ over $\Delta_n$.
and let $\zeta$ be a discrete random variable taking values in $\{1,\ldots,n\}$
distributed according to the {multinomial} distribution with $n$ categories and where the  probability of the $i$-th category is given by $x_i$. That is,
\begin{equation}\label{dist}
\mbox{Prob}(\zeta=i)=x_i \ (i=1,\ldots,n).
\end{equation}
Consider the random process:
$$y^{(0)}=0\in \mathbb{R}^n,\ y^{(r)}=y^{(r-1)}+e_{\zeta_r}\ (r\ge 1)$$
where $\zeta_r$ are independent random variables distributed according
to (\ref{dist}).
In other words, $y^{(r)}$ equals $y^{(r-1)}+e_i$ with probability $x_i$.
Finally, define
$$x^{(r)}={1\over r}y^{(r)} \in \Delta(n,r) \quad \ (r\ge 1).$$
\noindent
For a given $\alpha \in I(n,r)$, the probability of the event $y^{(r)} = \alpha$ is given by
\[
\mbox{Prob}(y^{(r)}=\alpha) = {r!\over \alpha!}x^\alpha,
\]
by the properties of the multinomial distribution.
Thus one also has
$
\mbox{Prob}(x^{(r)}=\alpha/r) = {r!\over \alpha!}x^\alpha,
$
%Hence, for any $\beta\in \oN^n$, the expected value of $(y^{(r)})^\beta$ is exactly the $\beta$-th moment of the multinomial distribution and, according to relations (\ref{relmom}) and (\ref{eqmom}), we obtain
%$$E((y^{(r)})^\beta)= m^\beta_{(n,r)}=r^{|\beta|} B_r(\phi_\beta)(x)$$ and thus
%$$E(f(y^{(r)})) = r^d B_r(f)(x)$$ for any $f\in \mathcal{H}_{n,d}$.
and it immediately follows that
\[
E(f(x^{(r)}))  =
\sum_{\alpha\in I(n,r)} \mbox{Prob}(x^{(r)}=\alpha/r) f(\alpha/r)
= \sum_{\alpha\in I(n,r)} {r!\over \alpha!}x^\alpha f\left({\alpha\over r}\right)
=
%{1\over r^d}  \sum_{\beta \in I(n,d)} f_\beta r^{|\beta|} B_r(\phi_\beta)(x) =
 B_r(f)(x).
\]
In this sense, the approach of our paper using Bernstein approximation is equivalent to the analysis of Nesterov \cite{Nes2003} (although the
equivalence is not obvious a priori).
%far from obvious!).
On the other hand, in \cite{Nes2003} the link with Bernstein approximation is not made, and the author
calculated the values $E(f(x^{(r)}))$ from first principles for polynomials up to degree four and for square-free polynomials. Based on this Nesterov \cite{Nes2003} gave the error bounds from Theorems \ref{berineq} and \ref{thmsf} for the quadratic and square-free cases. However he did not consider the general case.
%a few cases including quadratic and square-free forms.
Thus the analysis in this paper
  completes the analysis in \cite{Nes2003} by placing it in the well-studied framework of Bernstein approximation and clarifying the link to the multinomial distribution.

\newcommand{\MM}{\mathcal{M}}

\medskip
We conclude with a general comment regarding a further interpretation of  the upper bound $B_r(f)(x)$ (where $x\in \Delta_n$)  for the minimum $f_{\Delta(n,r)}$ over the regular grid, within the general framework introduced by Lasserre \cite{Las01} based on reformulating polynomial optimization problems as optimization problems over measures.
A basic, fundamental idea of Lasserre \cite{Las01} to compute the minimum of a polynomial $f$ over a compact set $K\subseteq \oR^n$ is to reformulate the problem as a minimization problem over the set $\MM(K)$ of (Borel) probability measures on the set $K$. (We assume $K$ compact for simplicity but Lasserre's idea works for $K$ closed). Namely,
$$\min_{x\in K} f(x)=\min _{\mu\in \MM(K)} E_\mu(f),$$
setting $E_\mu(f)=\int_K f(x)\mu(dx).$ The above identity is simple. As $f(x)\ge \min_{x\in K}f(x)$ for all $x\in K$, one can integrate both sides with respect to any measure $\mu\in \MM(K)$, which gives the inequality $\min_{x\in K} f(x){\le} \min _{\mu\in \MM(K)} E_\mu(f).$ For the converse inequality, let $\mu$ be the Dirac measure at a global minimizer $x$ of $f$ over $K$, so that $E_\mu(f)=\min_{x\in K}f(x)$.

Applying this idea to polynomial minimization over the regular grid $\Delta(n,r)$, one has
$$f_{\Delta(n,r)}=\min_{\mu\in \MM(\Delta(n,r))} E_\mu(f).$$
Thus in order to upper bound $f_{\Delta(n,r)}$ it suffices to choose a suitable probability measure on the regular grid $\Delta(n,r)$ and, according to our discussion above,  this is precisely what the bound $B_r(f)(x)$ boils down to.
% (where $x$ is chosen to be a global minimizer of $f$ over $\Delta_n$).

Indeed, by considering the multinomial distribution with $n$ categories and with probability $x_i$ for the $i$-th category, after $r$ independent trials we get a probability distribution over $I(n,r)$ where $\alpha\in I(n,r)$ is picked with probability ${r!\over \alpha!}x^\alpha$. This in turn gives a probability distribution $\mu_r$ on $\Delta(n,r)={1\over r} I(n,r)$ where $\alpha\over r$ is picked with the same probablity ${r!\over \alpha!}x^\alpha$.
Now, as was shown above,  $E_{\mu_r}(f) = B_r(f)(x)$ is thus an upper bound on $f_{\Delta(n,r)}$.

\medskip
%\tcolred{A final comment concerns the quality of the  approximation  $f_{\Delta(n,r)}$ of $\underline{f}$.
%The approach in this paper bounds in fact the maximum value of $B_r(f)-f$ over the simplex, which in turn upper bounds the desired range
% $f_{\Delta(n,r)}-\underline{f}$, as shown in Lemma \ref{claim1}.
%As observed in Examples  \ref{exquad} and \ref{exsf}
%the upper bound is tight  for
%$\max_{x\in \Delta_n} \{B_r(f)(x)-f(x)\}$ and  for  $\min_{x\in\Delta_n} B_r(f)(x)-\underline{f}$.
%However,  it is not tight for  $f_{\Delta(n,r)}-\underline{f}$.
%Moreover, while the upper bound is proportional to ${1\over r}(\overline{f}-\underline{f})$, it turns out that the range $f_{\Delta(n,r)}-\underline{f}$
%is proportional to ${1\over r^2}(\overline{f}- \underline{f})$ in Examples \ref{excubic}, \ref{exsf} and \ref{exdegd}. Yet it appears that the range $f_{\Delta(n,r)}-\underline{f}$ may depend linearly in $1\over r$, as shown for the quadratic polynomial considered in Example \ref{exquad} (recall relation (\ref{eqratior})). }
%%This raises the question of determining the exact rate of convergence of the parameters $f_{\Delta(n,r)}$ and whether or not it depends quadratically in $1\over r$.}

A final comment concerns the asymptotic convergence rate of the sequence
\[
\rho_r(f) = \frac{f_{\Delta(n,r)}-\underline{f}}{\overline{f} - \underline{f}} \quad r = 1,2,\ldots
\]
for a given polynomial $f\in \mathcal H_{n,d}$.
In all the examples presented in this paper, one has $\rho_r(f) = O(1/r^2)$; see Examples \ref{exquad}, \ref{excubic}, \ref{exsf} and \ref{exdegd}.
It remains an open problem to determine the asymptotic convergence rate in general.

\begin{acknowledgements}
%The authors would like to thank Immanuel Bomze for providing the reference \cite{JKB97},
%and Dima Pasechnik for suggesting the link with the approach by Nesterov \cite{Nes2003}.
%\tcolred{Did Dima suggest this?????}
{The first author thanks Immanuel Bomze for providing the reference \cite{JKB97},
and Dima Pasechnik for valuable comments on the approach by Nesterov \cite{Nes2003}.}
We also thank two anonymous reviewers for their comments that helped us improve the presentation of the paper.
\end{acknowledgements}

% \begin{thebibliography}
% Author, Article title, Journal, Volume, page numbers (year)
% Author, Book title, page numbers. Publisher, place (year)
% etc
% \end{thebibliography}

\appendix
%\section{The explicit formula of Bernstein approximation for a monomial over the simplex}
\section{Proof of Theorem \ref{theomom}}
\noindent In this Appendix, we give a self-contained proof for Theorem \ref{theomom}, which provides an explicit description  of the moments of the multinomial distribution in terms of the Stirling numbers of the second kind
(and as a direct application the explicit formulation for the Bernstein approximation on the simplex from   Corollary \ref{thmbernstein}).

Given $x\in \Delta_n$, we consider the multinomial distribution with $n$ categories and $r$ independent trials, where the probability of drawing the $i$-th category is given by $x_i$. Hence, given $\alpha\in I(n,r)$, the probability of drawing $\alpha_i$ times the $i$-th category for each $i\in [n]$ is equal to ${r!\over \alpha!}x^\alpha$ and, for any $\beta\in \oN^n$, the $\beta$-th moment of the multinomial distribution is
given by
\begin{equation}\label{eqmom0}
m^\beta_{(n,r)}:=\sum_{\alpha \in I(n,r)} \alpha^\beta {r!\over \alpha!} x^\alpha.
\end{equation}
Our objective is to show the following reformulaton of the $\beta$-th moment in terms of the Stirling numbers of the second kind:
\begin{equation}\label{eqmom1}
 m^\beta_{(n,r)} = \sum_{\alpha\in \oN^n: \alpha\le\beta}r^{\underline{|\alpha|}}x^{\alpha}\prod_{i=1}^nS(\beta_i,\alpha_i).
 \end{equation}

\noindent
Our proof is elementary in the sense that we will obtain  the moments of the multinomial distribution using its moment generating function.
One of the ingredients which we will use is the fact that the identity (\ref{eqmom1}) holds for the case $n=2$ of the binomial distribution when $\beta\in \oN^2$ is of the form $\beta=(\beta_1,0)$. Namely, the following identity is shown in \cite{AK08} (see Theorem 2.2 and relation (3.1) there).

\begin{lemma}\label{lemAK}\cite{AK08}
Given  $\beta_1\in \oN$ and  $x_1\in \oR$ such that $0\le x_1\le 1$, one has
$$m^{(\beta_1,0)}_{(2,r)}= \sum _{\alpha_1=0}^{r} \alpha_1^{\beta_1} {r\choose \alpha_1} x_1^{\alpha_1}(1-x_1)^{r-\alpha_1}
= \sum_{\alpha_1=0}^{\beta_1} r^{\underline{\alpha_1}} x_1^{\alpha_1} S(\beta_1,\alpha_1).$$
\end{lemma}

\noindent
This implies that the identity (\ref{eqmom1}) holds for  the moments of the multinomial distribution when the order $\beta$ has a single non-zero coordinate, i.e., $\beta $ is of the form $\beta=\beta_ie_i$ with $\beta_i\in\oN$.

\begin{corollary}
\label{lemS-Q}
Given $\beta_i\in \oN$ and $x\in \Delta_n$, one has
$$m^{\beta_ie_i}_{(n,r)} = \sum_{\alpha_i=0}^{\beta_i} r^{\underline{\alpha_i}} x_i^{\alpha_i} S(\beta_i,\alpha_i).$$
\end{corollary}

\begin{proof}
By (\ref{eqmom0}), we have
\begin{eqnarray*}
%B_r(\phi_{\beta_ie_i})(x)&=&\sum_{\alpha\in I(n,r)}({\alpha_i\over r})^{\beta_i}{r!\over \alpha!}x^{\alpha}\\
m^{(\beta_ie_i)}_{(n,r)} & =&\sum_{\alpha\in I(n,r)} \alpha_i^{\beta_i} {r!\over \alpha!} x^\alpha
= \sum_{\alpha_i=0}^r {\alpha_i}^{\beta_i}{r!\over \alpha_i!(r-\alpha_i)!}x_i^{\alpha_i}\left(\sum_{\overline{\alpha}\in I(n-1,r-\alpha_i)}{(r-\alpha_i)!\over {\overline{\alpha}}!}x^{\overline{\alpha}}\right)\\
&=& \sum_{\alpha_i=0}^r{\alpha_i}^{\beta_i}{r\choose \alpha_i} x_i^{\alpha_i}\left(\sum_{j\ne i} x_j\right)^{r-\alpha_i}
= \sum_{\alpha_i=0}^r{\alpha_i}^{\beta_i}{r \choose \alpha_i}x_i^{\alpha_i}(1-x_i)^{r-\alpha_i},\\
%&=& \sum_{\alpha_i=0}^r({\alpha_i\over r})^{\beta_i}{r\choose \alpha_i}x_i^{\alpha_i}(1-x_i)^{r-\alpha_i}\\
%&=&B_r^{Q}(\phi_{\beta_ie_i})(x).
%& = & \sum_{\alpha_i=0}^{\beta_i} r^{\underline{\alpha_i}} x_i^{\alpha_i} S(\beta_i,\alpha_i),\\
\end{eqnarray*}
which is equal to $ \sum_{\alpha_i=0}^{\beta_i} r^{\underline{\alpha_i}} x_i^{\alpha_i} S(\beta_i,\alpha_i)$ by  Lemma \ref{lemAK}.
 \qed \end{proof}

\noindent
In order to determine the moments of the multinomial distribution we use its moment generating function
$$t\in \oR^n \mapsto M_{x}^{r}(t):=\left(\sum_{i=1}^nx_ie^{t_i}\right)^r.$$
Then, for $\beta\in\oN^n$, the $\beta$-th moment of the multinomial distribution is equal to the $\beta$-th derivative of the moment {generating} function evaluated at $t=0$. Namely,
\begin{equation}\label{momentmgf}
m^{\beta}_{(n,r)}=\left.{\frac{\partial^{|\beta|}M_{x}^{r}(t)}{\partial t_1^{\beta_1}\cdots \partial t_{n}^{\beta_n} }}\right|_{t=0}.
\end{equation}
By Corollary  \ref{lemS-Q} we know that, for any $\beta_i\in \oN$,
\begin{equation}\label{re9}
m^{(\beta_ie_i)}_{(n,r)}=\left.{\frac{\partial^{\beta_i}M_{x}^{r}(t)}{\partial t_i^{\beta_i}}}\right|_{t=0}=\sum_{\alpha_i=0}^{\beta_i}S(\beta_i,\alpha_i)r^{\underline{\alpha_i}}x_i^{\alpha_i}.
\end{equation}
%The above relation (\ref{re9}) motivates the following result.
Next we show an analogue of the above relation (\ref{re9}) for  the evaluation of the $\beta_ie_i$-th derivative of the moment generating function at any point $t\in \oR^n$.
\begin{lemma}\label{lemsingle}
For  $x\in\Delta_n$, $\beta_i\in \oN$ and $t\in \oR^n$, one has
\begin{equation*}
{\frac{\partial^{\beta_i}M_{x}^{r}(t)}{\partial t_i^{\beta_i}}}
=\sum_{\alpha_i=0}^{\beta_i}S(\beta_i,\alpha_i)r^{\underline{\alpha_i}}x_i^{\alpha_i}e^{\alpha_it_i}M_{x}^{r-\alpha_i}(t).
\end{equation*}
\end{lemma}

\noindent
For the proof we will use the following recursive relation for the Stirling numbers of the second kind.
\begin{lemma}\label{lemrecurs}
For any integers $\beta\ge 0$ and $\alpha \ge 1$, one has
\begin{equation}\label{eq3}
S(\beta+1,\alpha)= S(\beta,\alpha -1) + \alpha S(\beta,\alpha).
\end{equation}
\end{lemma}

\begin{proof}
This well known fact can be easily checked as follows.
By definition, $S(\beta+1,\alpha)$ counts the number of ways of partitioning the set $\{1,\ldots,\beta,\beta+1\}$ into $\alpha$ nonempty subsets.
Considering the last element $\beta+1$, one can either put it in a singleton subset (so that there are $S(\beta,\alpha-1)$ such partitions), or partition $\{1,\ldots,\beta\}$ into $\alpha$ nonempty subsets and then assign the last element $\beta+1$ to one of them (so that there are $\alpha S(\beta,\alpha)$ such partitions). This shows the result.
\qed\end{proof}

\begin{proof} {\em (of Lemma \ref{lemsingle})}
To simplify notation we set  $M^r=M_{x}^{r}(t)$. We show the result using  induction on $\beta_i\ge 0$. The result holds clearly for $\beta_i=0$ and also for $\beta_i=1$ in which case we have
\begin{equation}\label{eq1}
\frac{\partial M^r}{\partial t_i}=rx_ie^{t_i}M^{r-1}.
\end{equation}
We now assume that the result holds for $\beta_i$ and we show that it also holds for $\beta_i+1$.
%Suppose the result holds for $\beta_i=k_i$, i.e.,
%\begin{equation}\label{eq2}
%{\frac{\partial^{k_i}M^{r}}{\partial t_i^{k_i}}}=\sum_{\alpha_i=0}^{k_i}S(k_i,\alpha_i)r^{\underline{\alpha_i}}x_i^{\alpha_i}e^{\alpha_it_i}M^{r-\alpha_i}.
%\end{equation}
%We will also use below the fact that $S(k_i,0)=0$ and $S(k_i,k_i)=1$.
%Then, for $\beta_i=k_i+1$, one has
For this, using the induction assumption, we obtain
\begin{equation}\label{eq000}
{\frac{\partial^{\beta_i+1}M^{r}}{\partial t_i^{{\beta_i}+1}}}
= { \partial\over \partial t_i}{ \partial^{{\beta_i}} M^{r}\over \partial t_i^{{\beta_i}}}
=
 { \partial\over \partial t_i}\left( \sum_{\alpha_i=0}^{\beta_i}S(\beta_i,\alpha_i)r^{\underline{\alpha_i}}x_i^{\alpha_i}e^{\alpha_it_i}M^{r-\alpha_i}\right)
 =  \sum_{\alpha_i=0}^{\beta_i}S(\beta_i,\alpha_i)r^{\underline{\alpha_i}}x_i^{\alpha_i}
  { \partial\over \partial t_i} ( e^{\alpha_it_i}M^{r-\alpha_i}).
   \end{equation}
   Now, using (\ref{eq1}), we can compute the last term as follows:
   $$ { \partial\over \partial t_i} ( e^{\alpha_it_i}M^{r-\alpha_i})
   = \alpha_i e^{\alpha_it_i} M^{r-\alpha_i} + (r-\alpha_i)x_i e^{(\alpha_i+1)t_i} M^{r-\alpha_i-1}.$$
   Plugging this into relation (\ref{eq000}), we deduce
 \begin{eqnarray*}
 {\frac{\partial^{\beta_i+1}M^{r}}{\partial t_i^{{\beta_i}+1}}}
 &=& \sum_{\alpha_i=0}^{\beta_i} \alpha_iS(\beta_i,\alpha_i) r^{\underline{\alpha_i}} x_i^{\alpha_i} e^{\alpha_it_i} M^{r-\alpha_i}
 +\sum_{\alpha_i=0}^{\beta_i} S(\beta_i,\alpha_i) r^{\underline {\alpha_i+1}} x_i^{\alpha_i+1} e^{(\alpha_i+1)t_i} M^{r-\alpha_i-1}\\
 &=&  \sum_{\alpha_i=0}^{\beta_i} \alpha_iS(\beta_i,\alpha_i) r^{\underline{\alpha_i}} x_i^{\alpha_i} e^{\alpha_it_i} M^{r-\alpha_i}
 +\sum_{\alpha_i'=1}^{\beta_i+1} S(\beta_i,\alpha'_i-1) r^{\underline{\alpha'_i}} x_i^{\alpha'_i} e^{\alpha'_it_i} M^{r-\alpha'_i}\\
&=& \sum_{\alpha_i=0}^{\beta_i} (\underbrace{\alpha_i S(\beta_i,\alpha_i)+ S(\beta_i,\alpha_i-1)}_{= S(\beta_i+1,\alpha_i) \ \text{ by } (\ref{eq3})})
r^{\underline{\alpha_i}} x_i^{\alpha_i} e^{\alpha_it_i} M^{r-\alpha_i}
+ r^{\underline{\beta_i+1}} x_i^{\beta_i+1} e^{(\beta_i+1)t_i}M^{r-\beta_i-1}\\
&=&\sum_{\alpha_i=0} ^{\beta_i+1} S(\beta_i+1,\alpha_i) r^{\underline{\alpha_i}} x_i^{\alpha_i} e^{\alpha_it_i} M^{r-\alpha_i},
   \end{eqnarray*}
  which concludes the proof.
%     \begin{eqnarray*}
% &=&\sum_{\alpha_i=0}^{k_i}
%S(k_i,\alpha_i)r^{\underline{\alpha_i}}x_i^{\alpha_i}e^{\alpha_it_i}{\partial M^{r-\alpha_i}\over\partial t_i}
%+S(k_i,\alpha_i)r^{\underline{\alpha_i}}x_i^{\alpha_i}\alpha_ie^{\alpha_it_i}M^{r-\alpha_i} \hfill \ \ \ \ \ \  \ \text{(by (\ref{eq2}))}\\
%&=& \sum_{\alpha_i=0}^{k_i}S(k_i,\alpha_i)r^{\underline{\alpha_i+1}}x_i^{\alpha_i+1}e^{(\alpha_i+1)t_i}M^{r-(\alpha_i+1)}+
%\sum_{\alpha_i=0}^{k_i}S(k_i,\alpha_i)r^{\underline{\alpha_i}}x_i^{\alpha_i}\alpha_ie^{\alpha_it_i}M^{r-\alpha_i}  \ \ \ \ \text{(by (\ref{eq1}))}\\
%
%&=& \sum_{\alpha'_i=1}^{k_i+1}S(k_i,\alpha'_i-1)r^{\underline{\alpha'_i}}x_i^{\alpha'_i}e^{\alpha'_it_i}M^{r-\alpha'_i}
%+ \sum_{\alpha_i=0}^{k_i}   S(k_i,\alpha_i)r^{\underline{\alpha_i}}x_i^{\alpha_i}\alpha_ie^{\alpha_it_i}M^{r-\alpha_i} \\
%
%&=& \sum_{\alpha_i=1}^{k_i}(S(k_i,\alpha_i-1)+\alpha_iS(k_i,\alpha_i)) r^{\underline{\alpha_i}}x_i^{\alpha_i}e^{\alpha_it_i}M^{r-\alpha_i}
%+ S(k_i,k_i)r^{\underline{k_i+1}} x_i^{k_i+1} e^{(k_i+1)t_i} M^{r-k_i-1}\\
%%
%&=& \sum_{\alpha_i=1}^{k_i} S(k_i+1,\alpha_i)  r^{\underline{\alpha_i}}x_i^{\alpha_i}e^{\alpha_it_i}M^{r-\alpha_i}
% + S(k_i+1,k_i+1)r^{\underline{k_i+1}} x_i^{k_i+1} e^{(k_i+1)t_i} M^{r-k_i-1}\
%\ \ \ \ \ \text{ (by (\ref{eq3}))} \\
%&=&
% \sum_{\alpha_i=0}^{k_i+1} S(k_i+1,\alpha_i)  r^{\underline{\alpha_i}}x_i^{\alpha_i}e^{\alpha_it_i}M^{r-\alpha_i}.\\
%\end{eqnarray*}
%This concludes the proof.
\qed \end{proof}
\noindent We now  extend the result of Lemma \ref{lemsingle} to an arbitrary derivative of the moment generating function.

\begin{theorem}\label{thmmoment}
For any $x\in\Delta_n$, $\beta\in \oN^n$ and $t\in \oR^n$, one has
$${\frac{\partial^{|\beta|}M_{x}^{r}(t)}{\partial t_1^{\beta_1}\cdots \partial t_{n}^{\beta_n} }}
=\sum_{\alpha\in \oN^n: \alpha\le\beta}r^{\underline{|\alpha|}}x^{\alpha}M_{x}^{r-|\alpha|}(t)
\left(\prod_{i=1}^ne^{\alpha_it_i}S(\beta_i,\alpha_i)\right).$$
\end{theorem}
\begin{proof}
We show the result using induction on the size $k$ of the support of $\beta$, i.e.,
$k=|\{i\in [n]:\beta_i\ne 0\}|.$
The result holds clearly for $k=0$ and, for  $k=1$,  the result holds by Lemma \ref{lemsingle}.
We now assume  that the result holds for $k$ and we show that it also holds for $k+1$. For this, consider the sequences $\beta'=(\beta_1,\ldots,\beta_k,\beta_{k+1},0,\ldots,0)$ and $\beta=(\beta_1,\ldots,\beta_k,0,0,\ldots,0)\in \oN^n$, where $\beta_1,\ldots,\beta_{k+1}\ge 1$.
By the induction assumption we know that
\begin{equation}\label{re10}
\frac{\partial^{|\beta|}M^r}{\partial t_1^{\beta_1}\cdots \partial t_{k}^{\beta_k}}=
\sum_{0\le\alpha\le\beta}r^{\underline{|\alpha|}}x^{\alpha}M^{r-|\alpha|}
\left(\prod_{i=1}^ne^{\alpha_it_i}S(\beta_i,\alpha_i)\right),
\end{equation}
setting  again $M^r=M^r_x(t)$ for simplicity.
%\noindent %Then, consider $\beta'=(\beta_1,\dots,\beta_{k+1},0,\dots,0)\in\oN^n$, we have:
Using (\ref{re10}), we obtain
\begin{eqnarray}\label{re98}
\frac{\partial^{|\beta'|}M^r}{\partial t_1^{\beta_1}\cdots \partial t_{k+1}^{\beta_{k+1}}}
&=&
{\partial^{\beta_{k+1}} \over \partial t_{k+1}^{\beta_{k+1}}}
\frac{\partial^{|\beta|}M^r}{\partial t_1^{\beta_1}\cdots \partial t_{k}^{\beta_{k}}}
=
{\partial^{\beta_{k+1}} \over \partial t_{k+1}^{\beta_{k+1}}}
\left(\sum_{0\le\alpha\le\beta}r^{\underline{|\alpha|}}x^{\alpha}M^{r-|\alpha|}
\left(\prod_{i=1}^ne^{\alpha_it_i}S(\beta_i,\alpha_i)\right) \right).
\end{eqnarray}
Note that $\alpha_{k+1}=0$ since $\alpha_{k+1}\le \beta_{k+1}$ and $\beta_{k+1}=0$.
 Hence, $M^{r-|\alpha|}$ is the only term containing the variable $t_{k+1}$ and thus (\ref{re98}) implies
  \begin{eqnarray}  \label{re99} %\ \ \ \text{(by (\ref{re10}))}\\
\frac{\partial^{|\beta'|}M^r}{\partial t_1^{\beta_1}\cdots \partial t_{k+1}^{\beta_{k+1}}}&=&\sum_{0\le\alpha\le\beta}r^{\underline{|\alpha|}}x^{\alpha}
\left(\prod_{i=1}^ne^{\alpha_it_i}S(\beta_i,\alpha_i)\right) {\partial^{\beta_{k+1}}M^{r-|\alpha|}\over \partial t_{k+1}^{\beta_{k+1}}}.
\end{eqnarray}
We now use  Lemma \ref{lemsingle} to compute the last term:
\begin{equation}\label{re100}
 {\partial^{\beta_{k+1}}M^{r-|\alpha|}\over \partial t_{k+1}^{\beta_{k+1}}}=
 \sum_{\theta_{k+1}=0}^{\beta_{k+1}}
S(\beta_{k+1},\theta_{k+1})(r-|\alpha|)^{\underline{\theta_{k+1}}}x_{k+1}^{\theta_{k+1}}e^{\theta_{k+1}t_{k+1}}M^{r-|\alpha|-\theta_{k+1}}.
 \end{equation}
Plugging (\ref{re100}) into (\ref{re99}) we obtain
\begin{eqnarray*}
&&\frac{\partial^{|\beta'|}M^r}{\partial t_1^{\beta_1}\cdots \partial t_{k+1}^{\beta_{k+1}}}\\
&=&\sum_{0\le\alpha\le\beta}r^{\underline{|\alpha|}}x^{\alpha}
\left(\sum_{\theta_{k+1}=0}^{\beta_{k+1}}
S(\beta_{k+1},\theta_{k+1})(r-|\alpha|)^{\underline{\theta_{k+1}}}x_{k+1}^{\theta_{k+1}}e^{\theta_{k+1}t_{k+1}}M^{r-|\alpha|-\theta_{k+1}}\right)
\left(\prod_{i=1}^ne^{\alpha_it_i}S(\beta_i,\alpha_i)\right)\\
&=&
\sum_{0\le\alpha\le\beta}
\sum_{\theta_{k+1}=0}^{\beta_{k+1}}
r^{\underline{|\alpha|+\theta_{k+1}}}x^{\alpha+e_{k+1}\theta_{k+1}}
S(\beta_{k+1},\theta_{k+1})e^{\theta_{k+1}t_{k+1}}M^{r-(|\alpha|+\theta_{k+1})}
\left(\prod_{i=1}^ne^{\alpha_it_i}S(\beta_i,\alpha_i)\right)\\
&=&
\sum_{0\le\alpha'\le\beta'}
r^{\underline{|\alpha'|}}x^{\alpha'}M^{r-|\alpha'|}
\left(\prod_{i=1}^ne^{\alpha_i't_i}S(\beta_i',\alpha_i')\right),
\end{eqnarray*}
after setting $\alpha'=\alpha+e_{k+1}\theta_{k+1}$.
This concludes the proof of Theorem \ref{thmmoment}. \qed \end{proof}

%\noindent Combining (\ref{momentmgf}) together with Theorem \ref{thmmoment} and recalling $B_r(\phi_{\beta})(x)={1\over r^{|\beta|}} m ^{\beta}_{(n,r)}$, we obtain the the result in Corollary \ref{thmbernstein}.

\end{document}